\RequirePackage{fix-cm}
\documentclass[smallextended]{svjour3}       
\smartqed  
\usepackage[main=english]{babel}

\usepackage{amsmath}
\usepackage{mathtools}
\usepackage{amsfonts}
\usepackage{hyperref}

\usepackage{todonotes}

\usepackage{xcolor}
\usepackage{enumitem}
\usepackage{tudacolors}

\usepackage{subcaption}

\DeclareMathOperator{\ima}{Im}
\newtheorem{assumption}{Assumption}

\begin{document}
	\title{Parareal for Higher Index Differential Algebraic Equations}
	
	\author{Idoia Cortes Garcia         \and
		Iryna Kulchytska-Ruchka			\and
		Sebastian Sch\"ops
	}
	
	\institute{Idoia Cortes Garcia, Iryna Kulchytska-Ruchka and Sebastian Sch\"ops \at
		Technische Universit\"at Darmstadt,  64289 Darmstadt, Germany\\
		Tel.: +49 6151 16 - 24392\\
		Fax: +49 6151 16 - 24404\\
		\email{idoia.cortes@tu-darmstadt.de}, \email{iryna.kulchytska-ruchka@tu-darmstadt.de}, 
		\email{sebastian.schoeps@tu-darmstadt.de}   
	}
	
	\date{Received: date / Accepted: date}

	\maketitle
	
	\begin{abstract}
		This article proposes modifications of the Parareal algorithm for its application to higher index differential algebraic equations (DAEs).
		It is based on the idea of applying the algorithm to only the differential components of the equation and the computation of corresponding
		consistent initial conditions later on. For differential algebraic equations with a special structure as e.g. given in flux-charge
		modified nodal analysis, it is shown that the usage of the implicit Euler method as a time integrator suffices for the Parareal
		algorithm to converge. Both versions of the Parareal method 
		 are applied to numerical examples of nonlinear index 2 differential algebraic equations.
	\end{abstract}

	\section{Introduction}
	The time-domain simulation of models from physics, finance or social sciences often leads to differential and algebraic
	equations systems.
	The systems of equations describing the transient behaviour of the required physical quantities can be both time dependent
	differential equations, such as e.g. in circuit models for the simulation of microchips and energy networks \cite{Falgout_2017ab}, or space and time dependent
	partial differential equations (PDEs), as is the case for the simulation of the electromagnetic behaviour of 
	electric machines \cite{Salon_1995aa}.
	In the latter one, typically the method of lines is used, where first a spatial discretisation method is applied to the PDE
	to obtain an only time dependent system of differential equations which is then solved in time as an initial value problem (IVP). 
	
	The time domain simulation of large systems of equations e.g. obtained from fine meshes as well as 
	fast dynamics of the excitations which require small time
	step sizes together with large time windows
	  considerably increase simulation time. In these cases, parallelisation methods allow reducing computation time. 
	  When spatial parallelisation
	  by means of domain decomposition methods is used up to saturation or whenever the time domain dynamics is the bottleneck of the simulation
	  time, parallel-in-time methods \cite{Nievergelt_1964aa,Lions_2001aa,Gander_2015aa,Takahashi_2019aa} can be used. 
	  Parareal is such an algorithm which is based on the same idea as multiple shooting 
	  methods \cite{Lions_2001aa,Gander_2015aa}. 
	  
	Many initial value problems arising from physical systems such as e.g. electric networks, spatial discretisation of some approximations to 
	Maxwell's equations or constrained mechanical systems such as the pendulum are systems of differential algebraic equations (DAEs). 
	These are systems that contain both ordinary differential equations (ODEs) as well as algebraic constraints. They convey
	analytical and numerical difficulties that do not arise when handling ODEs. This includes their potential large sensitivity towards
	small high frequent perturbations \cite{Brenan_1995aa,Lamour_2013aa} as well as their non-trivial selection of appropriate initial conditions 
	\cite{Lamour_2013aa,Estevez-Schwarz_2000ab}. One way of
	 classifying them according to the difficulties they pose is given by their  index, a natural number $\geq 0$.  Especially when using less
	  standard algorithms,
	such as e.g. Parareal, on DAEs with higher index, 
	a correct handling of the equations is of utmost importance. This work focuses on the application of the
	Parareal algorithm to index 2 DAEs. Low-index Problems have already been numerically solved by Parareal, e.g. \cite{Cadeau_2011aa}. 
	The rather straight-forward index-1 case was already discussed in \cite{Schops_2018aa,Falgout_2019aa,Falgout_2017ab}.

	The paper is structured as follows: Section~\ref{sec:dae} introduces the concept of differential algebraic equations, the tractability index
	and the problems that arise for the choice of consistent initial conditions. A first result is given for the behaviour of the implicit
	Euler scheme on DAEs with a specific structure.  In Section~\ref{sec:parareal} the classic Parareal algorithm is presented and a
	 modification of the algorithm
	is proposed for its application to index 2 tractable DAEs. 
	The (possibly) index 2 DAE for circuit simulation arising from the modified nodal
	analysis is presented and
	Parareal is applied to two  nonlinear index~2 DAEs in Section~\ref{sec:numerics}. 
	The paper concludes in Section~\ref{sec:conclusions} with a summary.

	\section{Differential Algebraic Equations}\label{sec:dae}	
	We consider initial value problems consisting of quasilinear differential algebraic equations of the form
	\begin{equation}\label{eq:quaslinDAE}
	\mathbf{A}(\mathbf{x},t)\mathbf{x}' + \mathbf{b}(\mathbf{x},t) = 0\;,
	\end{equation}
	with $\mathbf{x}:\mathcal{I}\rightarrow\mathbb{R}^{n_{\mathrm{dof}}}$, where $\mathcal{I} = [t_0\; t_{\mathrm{end}}]\subset \mathbb{R}$ is a
	time interval and $n_{\mathrm{dof}}$ the number of degrees of freedom and initial condition $\mathbf{x}(t_0) = \mathbf{x}_0$. 
	Here, $\det \mathbf{A}(\mathbf{x},t)$ can be zero, which yields a system of equations containing both differential
	equations as well as algebraic constraints.
	
	Differential algebraic equations are typically classified according to their index \cite{Brenan_1995aa,Hairer_1996aa,Lamour_2013aa}, 
	which allows evaluating
	the analytical and numerical difficulties the system may convey. Higher index systems (index $\geq 2$) require a 
	special numerical handling. There are different types of index definitions, that essentially
	coincide for linear systems \cite{Mehrmann_2015aa}.
	In this paper we introduce the projector-based tractability index \cite{Lamour_2013aa} as it allows a separation 
	of the degrees of freedom and equations into the purely differential and the algebraic ones. 
	
	Analogously to \cite{Estevez-Schwarz_2000ab,Cortes-Garcia_2020ae}, we consider DAEs with a special structure and 
	therefore take the following assumption.
	\begin{assumption}[Mass matrix]\label{ass:mass}
		We assume the spaces $\ker \mathbf{A}(\mathbf{x},t)$ and \\
		$\ima \mathbf{A}(\mathbf{x},t)$ are independent of the
		degrees of freedom $\mathbf{x}$ and depend smoothly on $t$.
	\end{assumption}

	\begin{remark}
		This assumption is mild, as many  DAEs arising from physical systems fulfil those requirements. Later on, 
		it is shown that e.g. RLC circuits described with modified nodal analysis have this property 
		(see \cite{Estevez-Schwarz_2000aa,Estevez-Schwarz_2000ab}) or even
		space discretised partial differential algebraic equations such as e.g. the eddy current problem \cite{Schops_2011ac}. Similar assumptions are taken e.g. in \cite{Lamour_1997aa} and ensure the BDF method integrates index 2 problems well.
	\end{remark}

	Let us consider quasilinear DAEs fulfilling Assumption~\ref{ass:mass} and their corresponding projectors $\mathbf{Q}(t)$ onto 
	$\ker \mathbf{A}(\mathbf{x},t)$ as well as its complementary $\mathbf{P}(t) = \mathbf{I} - \mathbf{Q}(t)$. We
	introduce the matrices 
	\begin{align*}
		\mathbf{B}(\mathbf{y}, \mathbf{x},t) &\coloneqq \frac{\partial }{\partial \mathbf{x}} \left(\mathbf{A}(\mathbf{x},t)\mathbf{y}\right)
														+ \frac{\partial }{\partial \mathbf{x}} \mathbf{b}(\mathbf{x},t)\\
		\mathbf{A}_1(\mathbf{y}, \mathbf{x},t) &\coloneqq \left(\mathbf{A}(\mathbf{x},t) + 
																\mathbf{B}(\mathbf{y}, \mathbf{x},t)\mathbf{Q}(t) \right)
														\left(\mathbf{I} -  \mathbf{P}(t)\mathbf{P}'(t)\mathbf{Q}(t)\right)\,,
	\end{align*}
	and the projectors $\mathbf{Q}_1(\mathbf{y}, \mathbf{x},t)$ onto $\ker \mathbf{A}_1(\mathbf{y}, \mathbf{x},t)$ and 
	 $\mathbf{P}_1(\mathbf{y}, \mathbf{x},t) = \mathbf{I} - \mathbf{Q}_1(\mathbf{y}, \mathbf{x},t)$. Finally, the matrix
	\begin{align}\label{eq:G2}
		\mathbf{G}_2(\mathbf{y}, \mathbf{x},t) \coloneqq \mathbf{A}_1(\mathbf{y}, \mathbf{x},t) + 
														\mathbf{B}(\mathbf{y}, \mathbf{x},t)\mathbf{P}(t)\mathbf{Q}_1(\mathbf{y}, \mathbf{x},t)
	\end{align}
	is defined.
	As we only focus on index 2 systems, we present the tractability index definition accordingly. However,
	it can be generalised to systems with index $> 2$ (see \cite{Lamour_2013aa}).  
	\begin{definition}[Tractability index \cite{Lamour_2013aa}]
		A quasilinear DAE \eqref{eq:quaslinDAE} fulfilling Assumption~\ref{ass:mass} has tractability index
		\begin{itemize}[label=\textbullet]
			\item 0, if $\mathbf{A}(\mathbf{x},t)$ is nonsingular,
			\item 1, if $\mathbf{A}(\mathbf{x},t)$ is singular and $\mathbf{A}_1(\mathbf{y}, \mathbf{x},t)$ is nonsingular, 
			\item 2, if $\mathbf{A}(\mathbf{x},t)$ and $\mathbf{A}_1(\mathbf{y}, \mathbf{x},t)$ are singular and
					$\mathbf{G}_2(\mathbf{y}, \mathbf{x},t)$ is nonsingular.
		\end{itemize}
	\end{definition}
	
	For the rest of the paper, we will choose a specific projector $\mathbf{Q}_1$, which fulfils a property that is especially helpful 
	in the setting of the implicit Euler method. This particular choice can be taken without the loss of generality 
	(see~\cite{Estevez-Schwarz_2000ab}).
	\begin{assumption}[Canonical projector \cite{Estevez-Schwarz_2000ab}] \label{ass:canonical}
		We define $\mathbf{Q}_1$ to be the canonical projector, that is, for a projector 
		$\mathbf{\tilde{Q}}_1(\mathbf{y}, \mathbf{x},t)$ onto $\ker \mathbf{A}_1(\mathbf{y}, \mathbf{x},t)$, we choose
		\begin{equation*}
			\mathbf{Q}_1(\mathbf{y}, \mathbf{x},t)  = \mathbf{\tilde{Q}}_1(\mathbf{y}, \mathbf{x},t)
												\mathbf{G}_2(\mathbf{y}, \mathbf{x},t)^{-1}
												\mathbf{B}(\mathbf{y}, \mathbf{x},t)\mathbf{P}(t)\;.							
		\end{equation*}
	\end{assumption}
	Note that here the projector fulfills the following property \cite{Estevez-Schwarz_2000ab}
	\begin{equation*}
		\mathbf{Q}_1(\mathbf{y}, \mathbf{x},t)\mathbf{Q}(t) = 0\;.
	\end{equation*} 

	\subsection{Consistent Initialisation}
	 
	 In contrast to initial value problems arising from ODEs, DAEs cannot be initialised with arbitrary initial conditions
	 $\mathbf{x}_0$, as the algebraic constraints imposed by the system have to be fulfilled by the solution at every time
	 point $t\in \mathcal{I}$. In this context we introduce the concept of consistent initial conditions.
	 \begin{definition}[Consistent initial condition c.f. \cite{Brenan_1995aa}]
	 	Let us consider an initial value problem consisting of the quasilinear DAE \eqref{eq:quaslinDAE} on the time
	 	interval $t\in \mathcal{I}$. Then, an initial condition $\mathbf{x}_0$ at $t_0$ is called consistent, if there
	 	exists a solution $\mathbf{x}^*:\mathcal{I}\rightarrow\mathbb{R}^{n}$ of \eqref{eq:quaslinDAE} fulfilling
	 	$\mathbf{x}^*(t_0) = \mathbf{x}_0$.
	 \end{definition}
	 \begin{remark}
	 	Higher index DAEs ($\geq 2$) present hidden constraints that appear only after time differentiation of the original
	 	system and thus are not explicitly accessible This complicates the choice of appropriate initial conditions,
	 	as they have to fulfil both the explicit as well as the hidden constraints.
	 \end{remark}
 	Following \cite{Estevez-Schwarz_2000ab,Lamour_2013aa,Marz_1994ab}, we make usage of the projectors defined for the tractability index to 
 	separate the degrees of freedom of the DAE system and extract the purely differential components. Their ICs are freely choosable
 	as they are not characterised by any type of algebraic constraint. Furthermore, when they are
 	prescribed together with the quasilinear DAE \eqref{eq:quaslinDAE},  a uniquely solvable initial value problem
 	arises \cite{Marz_1994ab}. They can be extracted with 
 	$$\mathbf{x}_{\mathrm{diff}}\coloneqq\mathbf{P}\mathbf{P}_1(\mathbf{y},\mathbf{x},t)\mathbf{x}$$ 
 	and, if they are fixed, the rest of the components, all of them 
 	algebraic, are uniquely determined by the values of $\mathbf{x}_{\mathrm{diff}}$ and time 
 	$t$ \cite{Estevez-Schwarz_2000ab}. Note that, for index~2 systems, two projectors are required to extract the differential components. 
 	Here a further reduction within the explicitly differentiated $\mathbf{P}\mathbf{x}$ components is made by further applying the 
 	$\mathbf{P}_1(\mathbf{y},\mathbf{x},t)$ projector. 
 	
 	With a second projector $\mathbf{T}(\mathbf{y}, \mathbf{x},t)$ onto 
 	$\ima \mathbf{Q}(t)\mathbf{Q}_1(\mathbf{y}, \mathbf{x},t)$, the index 2 variables of the DAE can be extracted.
 	For the differential index, the index 2 variables are the parts of $\mathbf{x}$ that require one time differentiation of the original 
 	system to be characterised by 
 	(hidden) algebraic constraints \cite{Brenan_1995aa}. Its complementary projector
 	$\mathbf{U}(\mathbf{y}, \mathbf{x},t) = \mathbf{I} - \mathbf{T}(\mathbf{y}, \mathbf{x},t)$ allows to extract the
 	purely differential components together with the index 1 variables, that is, the ones prescribed by the explicit
 	constraints of the system. Thus we can separate the degrees of freedom into three types (see \cite{Estevez-Schwarz_2000ab})
 	\[
 	\mathbf{x} = \underbrace{\mathbf{P}\mathbf{P}_1(\mathbf{y},\mathbf{x},t)\mathbf{x}}_{\text{index 0}} + \underbrace{\mathbf{P}\mathbf{Q}_1(\mathbf{y},\mathbf{x},t)\mathbf{x} + \mathbf{Q}\mathbf{U}(\mathbf{y},\mathbf{x},t)\mathbf{x}}_{\text{index 1}}
 	+ \underbrace{\mathbf{T}(\mathbf{y}, \mathbf{x},t)\mathbf{x}}_{\text{index 2}}\;.
 	\]

 	\subsection{Implicit Euler}
 	Given a quasilinear DAE \eqref{eq:quaslinDAE} defined on the time interval $\mathcal{I}$ with consistent initial condition $\mathbf{x}_0$ at initial 
 	time $t_0$, it can numerically be integrated with the implicit Euler method. 
 	For time steps $t_0,t_1,\ldots,t_n$ with $t_n=t_{\mathrm{end}}$, step size $t_{i+1}-t_i=h$ and approximated solutions
 	$\mathbf{x}_0,\ldots,\mathbf{x}_i$, the implicit 
 	Euler method performs for the $(i+1)$th time step the approximation
 	$$ \mathbf{A}(\mathbf{x}_{i+1},t_{i+1})\frac{\mathbf{x}_{i+1}-\mathbf{x}_{i}}{h} + \mathbf{b}(\mathbf{x}_{i+1},t_{i+1}) = 0\;.$$
 	
 	In the following section we consider DAEs with a simplified structure. For that we impose some additional requirements
 	on the index 2 components.
 	
 	\begin{assumption}[Constant projector matrices \cite{Estevez-Schwarz_2000ab}]\label{ass:constT}
 		We assume the mass matrix $\mathbf{A}$ and the space 
 		$\ima \mathbf{Q}(t)\mathbf{Q}_1(\mathbf{y}, \mathbf{x},t)$ are time and space independent
 		and thus the projectors $\mathbf{P}$, $\mathbf{Q}$, $\mathbf{T}$ and $\mathbf{U}$ are constant. 	
 		Without loss of generality (see \cite{Estevez-Schwarz_2000ab}), we consider the projector $\mathbf{T}$ to 
 		fulfill $\mathbf{T}\mathbf{P} = 0$.
 	\end{assumption}
 	Now we study the behaviour of the implicit Euler scheme applied on index 2 tractable DAEs with
 	linear index 2 components and constant mass matrix, that is, systems written as 
 	\begin{equation}\label{eq:DAElinindex2}
 		\mathbf{A}\mathbf{x}' + \mathbf{b}_1(\mathbf{U}\mathbf{x},t) + \mathbf{B}_2 \mathbf{T}\mathbf{x}=0\;.
 	\end{equation}
 	
 	\begin{assumption}[Constant $\mathbf{Q}_1$]\label{ass:Q1ct}
 		We assume the index 2 tractable DAE in \eqref{eq:DAElinindex2} has a constant projector $\mathbf{Q}_1^*$
 		onto $\ima \mathbf{Q}_1(\mathbf{U}\mathbf{x},t)$ and define $\mathbf{P}_1^* = \mathbf{I} - \mathbf{Q}_1^*$.
 	\end{assumption}
 	\begin{remark}
 		We will show later nonlinear index 2 DAEs arising from application examples that fulfil the structural assumptions
 		made in this section to demonstrate that they are not too restrictive.
 	\end{remark}
 	Before discussing several special properties of the implicit Euler scheme, some important characteristics
 	 of the projectors are presented.
 	\begin{proposition}[Projectors c.f. \cite{Estevez-Schwarz_2000ab}] \label{prop:projectors}
 		We consider a DAE as in \eqref{eq:DAElinindex2} with its corresponding projectors 
 		 fulfilling Assumptions \ref{ass:canonical}, \ref{ass:constT} and \ref{ass:Q1ct}.
 		Then it holds that
 		\begin{enumerate}[label=(\roman*),leftmargin=1cm]
 			\item $\mathbf{P}\mathbf{P}_1(\mathbf{U}\mathbf{x},t)\mathbf{P} = \mathbf{P}\mathbf{P}_1(\mathbf{U}\mathbf{x},t)\;,$
 			\item $\mathbf{U}\mathbf{Q}\mathbf{P}_1(\mathbf{U}\mathbf{x},t)\mathbf{P} = 0\;,$
 			\item $\mathbf{G}_2(\mathbf{U}\mathbf{x},t)^{-1}\mathbf{A} = \mathbf{P}_1(\mathbf{U}\mathbf{x},t)\mathbf{P}\;,$
 			\item $\mathbf{G}_2(\mathbf{U}\mathbf{x},t)^{-1}\mathbf{B}_2\mathbf{T} = \mathbf{T}\;,$
 			\item $\mathbf{P}\mathbf{P}_1(\mathbf{U}\mathbf{x},t) = 
 			\mathbf{P}\mathbf{P}_1(\mathbf{U}\mathbf{x},t)\mathbf{P}\mathbf{P}_1^*\;,$
 		\end{enumerate}
 		where the matrix $\mathbf{G}_2(\mathbf{U}\mathbf{x},t)$ is defined as \eqref{eq:G2}.
 	\end{proposition}
 	\begin{proof}
 		Property $(i)$ follows from the feature of the canonical projector shown in Assumption \ref{ass:canonical}
 		and $(ii)$ from the definition of the projector $\mathbf{T}$. In \cite[Chapter~2.3]{Estevez-Schwarz_2000ab} Property 
 		$(iii)$ is shown and proven
 		and for $(iv)$ the equivalent expression $\mathbf{B}_2\mathbf{T} = \mathbf{G}_2(\mathbf{U}\mathbf{x},t)\mathbf{T}$
 		is derived by applying the definitions of the projectors, Assumption~\ref{ass:constT} and $(i)$. Finally,
 		$(v)$ follows from the definition of $\mathbf{Q}_1^*$   that implies 
 		$\mathbf{P}_1(\mathbf{U}\mathbf{x},t)\mathbf{P}_1^* = \mathbf{P}_1(\mathbf{U}\mathbf{x},t)$ 
 		(see \cite[Appendix~A.1]{Baumanns_2012ab}). \qed
 	\end{proof}
 				
 	\begin{proposition}[Implicit Euler consistentialisation]\label{pro:impleul}
 		We consider a DAE and the corresponding projectors fulfilling Assumptions \ref{ass:canonical}, \ref{ass:constT} and \ref{ass:Q1ct}, 
 		with structure as in \eqref{eq:DAElinindex2}
 		and two initial conditions at $t_0$, the first one, $\mathbf{x}^0$, being inconsistent 
 		\begin{equation}\label{eq:inconsIC}
 			\mathbf{x}(t_0) = \mathbf{x}^0
 		\end{equation}
 		 and the second one, $\mathbf{x}_0$,
 		being consistent 
 		\begin{equation}\label{eq:consIC}
 		\mathbf{x}(t_0) = \mathbf{x}_0\;.
 		\end{equation}
 		Let $\mathbf{x}^2$ be the solution obtained at time point $t_2$ after two implicit Euler steps of the
 		IVP \eqref{eq:DAElinindex2} with inconsistent IC \eqref{eq:inconsIC} and
 		$\mathbf{x}_2$ the one obtained with consistent IC \eqref{eq:consIC} both with the same time step sizes.
 	 	If both initial conditions are such that
 		\begin{equation*}
 			\mathbf{P}\mathbf{P}_1^*\mathbf{x}^0 = \mathbf{P}\mathbf{P}_1^*\mathbf{x}_0\;,
 		\end{equation*}
 		then $\mathbf{x}^2 = \mathbf{x}_2$.
 	\end{proposition}
 	\begin{proof}
 		The proof proceeds similarly to the approach taken in \cite[Chapter~2.5]{Estevez-Schwarz_2000ab} for the computation of consistent initial
 		conditions for index 2 DAEs. The superscript $i$ is used to denote
 		the solutions of the implicit Euler method for time step $t_i$ starting with inconsistent initial condition $\mathbf{x}^0$ 
 		and subscript $i$ for the solutions with consistent initial condition $\mathbf{x}_i$. 
 		
 		Application of the implicit Euler scheme
 		yields for the first time step with step size $h$
 		\begin{align*}
 			\mathbf{A}\frac{\mathbf{x}^1-\mathbf{x}^0}{h} + \mathbf{b}_1(\mathbf{U}\mathbf{x}^1,t_1)
 			+ \mathbf{B}_2\mathbf{T}\mathbf{x}^1 &= 0\\
 			\mathbf{A}\frac{\mathbf{x}_1-\mathbf{x}_0}{h} + \mathbf{b}_1(\mathbf{U}\mathbf{x}_1,t_1)
 			+ \mathbf{B}_2\mathbf{T}\mathbf{x}_1 &= 0\;.
 		\end{align*}
 		The equations are then multiplied by $\mathbf{G}_2(\mathbf{U}\mathbf{x}^1,t_1)^{-1}$ and 
 		$\mathbf{G}_2(\mathbf{U}\mathbf{x}_1,t_1)^{-1}$, respectively, which leads to
 		\begin{align*}
 			\mathbf{P}_1(\mathbf{U}\mathbf{x}^1,t_1)\mathbf{P}\frac{\mathbf{x}^1-\mathbf{x}^0}{h}
 			+ \mathbf{G}_2(\mathbf{U}\mathbf{x}^1,t_1)^{-1}\mathbf{b}_1(\mathbf{U}\mathbf{x}^1,t_1) 
 			+ \mathbf{T}\mathbf{x}^1 &=0\\
 			\mathbf{P}_1(\mathbf{U}\mathbf{x}_1,t_1)\mathbf{P}\frac{\mathbf{x}_1-\mathbf{x}_0}{h}+
 			\mathbf{G}_2(\mathbf{U}\mathbf{x}_1,t_1)^{-1}\mathbf{b}_1(\mathbf{U}\mathbf{x}_1,t_1)+
 			\mathbf{T}\mathbf{x}_1 &= 0
 		\end{align*}
 		due to Proposition~\ref{prop:projectors} $(iii)$ and $(iv)$. 
 		
 		Now we multiply both equations by $\mathbf{T}$ and $\mathbf{U}$ to split them into the parts
 		defining $\mathbf{T}\mathbf{x}^1$ (respectively $\mathbf{T}\mathbf{x}_1$) and 
 		$\mathbf{U}\mathbf{x}^1$ (respectively $\mathbf{U}\mathbf{x}_1$). Thus we have
 		\begin{align*}
 		\mathbf{T}\mathbf{P}_1(\mathbf{U}\mathbf{x}^1,t_1)\mathbf{P}\frac{\mathbf{x}^1-\mathbf{x}^0}{h}
 		+ \mathbf{T}\mathbf{G}_2(\mathbf{U}\mathbf{x}^1,t_1)^{-1}\mathbf{b}_1(\mathbf{U}\mathbf{x}^1,t_1) 
 		+ \mathbf{T}\mathbf{x}^1&=0\\
 			\mathbf{U}\mathbf{P}_1(\mathbf{U}\mathbf{x}^1,t_1)\mathbf{P}\frac{\mathbf{x}^1-\mathbf{x}^0}{h}
 			+ \mathbf{U}\mathbf{G}_2(\mathbf{U}\mathbf{x}^1,t_1)^{-1}\mathbf{b}_1(\mathbf{U}\mathbf{x}^1,t_1) &=0
 		\end{align*}
 		and
 		\begin{align*}
 		\mathbf{T}\mathbf{P}_1(\mathbf{U}\mathbf{x}_1,t_1)\mathbf{P}\frac{\mathbf{x}_1-\mathbf{x}_0}{h}
 		+ \mathbf{T}\mathbf{G}_2(\mathbf{U}\mathbf{x}_1,t_1)^{-1}\mathbf{b}_1(\mathbf{U}\mathbf{x}_1,t_1) 
 		+ \mathbf{T}\mathbf{x}_1&=0\\
 		\mathbf{U}\mathbf{P}_1(\mathbf{U}\mathbf{x}_1,t_1)\mathbf{P}\frac{\mathbf{x}_1-\mathbf{x}_0}{h}
 		+ \mathbf{U}\mathbf{G}_2(\mathbf{U}\mathbf{x}_1,t_1)^{-1}\mathbf{b}_1(\mathbf{U}\mathbf{x}_1,t_1) &=0\;.
 		\end{align*}
 		Application of Proposition \ref{prop:projectors} $(ii)$, $(i)$ and $(v)$ together with equality
 		$\mathbf{P}\mathbf{P}_1^*\mathbf{x}^0 = \mathbf{P}\mathbf{P}_1^*\mathbf{x}_0$ to the equations
 		defining the $\mathbf{U}$ components, and $\mathbf{P}\mathbf{T} = 0$ yields
 		\begin{align}
 		\mathbf{T}\mathbf{P}_1(\mathbf{U}\mathbf{x}^1,t_1)\mathbf{P}\mathbf{U}\frac{\mathbf{x}^1-\mathbf{x}^0}{h}
 		+ \mathbf{T}\mathbf{G}_2(\mathbf{U}\mathbf{x}^1,t_1)^{-1}\mathbf{b}_1(\mathbf{U}\mathbf{x}^1,t_1) 
 		+ \mathbf{T}\mathbf{x}^1&=0\label{eq:Tincons}\\
 		\mathbf{U}\mathbf{P}\mathbf{P}_1(\mathbf{U}\mathbf{x}^1,t_1)\frac{\mathbf{x}^1-\mathbf{x}_0}{h}
 		+ \mathbf{U}\mathbf{G}_2(\mathbf{U}\mathbf{x}^1,t_1)^{-1}\mathbf{b}_1(\mathbf{U}\mathbf{x}^1,t_1) &=0\label{eq:Uincons}
 		\end{align}
 		and
 		\begin{align}
 		\mathbf{T}\mathbf{P}_1(\mathbf{U}\mathbf{x}_1,t_1)\mathbf{P}\mathbf{U}\frac{\mathbf{x}_1-\mathbf{x}_0}{h}
 		+ \mathbf{T}\mathbf{G}_2(\mathbf{U}\mathbf{x}_1,t_1)^{-1}\mathbf{b}_1(\mathbf{U}\mathbf{x}_1,t_1) 
 		+ \mathbf{T}\mathbf{x}_1&=0\label{eq:Tcons}\\
 		\mathbf{U}\mathbf{P}\mathbf{P}_1(\mathbf{U}\mathbf{x}_1,t_1)\frac{\mathbf{x}_1-\mathbf{x}_0}{h}
 		+ \mathbf{U}\mathbf{G}_2(\mathbf{U}\mathbf{x}_1,t_1)^{-1}\mathbf{b}_1(\mathbf{U}\mathbf{x}_1,t_1) &=0\;.\label{eq:Ucons}
 		\end{align}
 		Here, \eqref{eq:Uincons} and \eqref{eq:Ucons} are to be solved to obtain the solutions for $\mathbf{U}\mathbf{x}^1$ and
 		$\mathbf{U}\mathbf{x}_1$, respectively. It can be seen that both equations are equivalent and their solution only depends
 		 on $\mathbf{P}\mathbf{P}_1^*\mathbf{x}_0$.
 		Therefore they yield the same solution, i.e. $\mathbf{U}\mathbf{x}^1 = \mathbf{U}\mathbf{x}_1$.
 		
 		To obtain $\mathbf{T}\mathbf{x}^1$ and $\mathbf{T}\mathbf{x}_1$, \eqref{eq:Tincons} and \eqref{eq:Tcons} have to be solved.
 		Again both equations are equivalent, however, whereas the first one depends on $\mathbf{U}\mathbf{x}^1$ and $\mathbf{U}\mathbf{x}^0$,
 		the second one depends on $\mathbf{U}\mathbf{x}_1$ and $\mathbf{U}\mathbf{x}_0$.  Due to the previous step, 
 		$\mathbf{U}\mathbf{x}^1 = \mathbf{U}\mathbf{x}_1$, but,  $\mathbf{U}\mathbf{x}^0$ is not necessarily equal to $\mathbf{U}\mathbf{x}_0$.
 		Thus, if $\mathbf{U}\mathbf{x}^0 = \mathbf{U}\mathbf{x}_0$, then $\mathbf{T}\mathbf{x}^1=\mathbf{T}\mathbf{x}_1$ and only
 		one implicit Euler step is required to obtain the same solution with both initial conditions $\mathbf{x}^0$ and $\mathbf{x}_0$.
 		
 		If $\mathbf{U}\mathbf{x}^0 \neq \mathbf{U}\mathbf{x}_0$, then the analogous procedure is repeated to obtain the solution
 		for $t_2$.  This time, the index 2 components $\mathbf{T}\mathbf{x}^2$ and 
 		$\mathbf{T}\mathbf{x}_2$ are again defined by the same equation and depend on $\mathbf{U}\mathbf{x}^1$ and $\mathbf{U}\mathbf{x}^2$
 		or $\mathbf{U}\mathbf{x}_1$ and $\mathbf{U}\mathbf{x}_2$, respectively. Due to the previous step, $\mathbf{U}\mathbf{x}^1 = \mathbf{U}\mathbf{x}_1$ and, analogously as before, we obtain
 		$\mathbf{U}\mathbf{x}^2 = \mathbf{U}\mathbf{x}_2$ and thus $\mathbf{x}_2 = \mathbf{x}^2$. \qed
 	\end{proof}
 
 \begin{remark}
 	Note that, as the implicit Euler scheme starting with a consistent initial condition yields solutions that are
 	consistent for semiexplicit index 2 DAEs, $\mathbf{x}_2$ 
 	is consistent \cite{Brenan_1995aa}. 
 	Therefore, as $\mathbf{x}^2=\mathbf{x}_2$, $\mathbf{x}^2$ is
 	a consistent solution that is obtained after two implicit Euler steps with an inconsistent
 	initial condition. In practice this consistency is only obtained up to a certain tolerance which depends
 	e.g. on the accuracy of the Newton scheme.
 \end{remark}
	 In \cite{Estevez-Schwarz_2000ab} a similar result is shown for the implicit Euler scheme. However, for schemes starting with only inconsistent 
	 index 2 variables, that is, schemes where $\mathbf{U}\mathbf{x}^0 = \mathbf{U}\mathbf{x}_0$ (in fact, only 
	 $\mathbf{P}\mathbf{x}^0 = \mathbf{P}\mathbf{x}_0$ is required). We have extended this result to consider DAEs, where 
	 also the $\mathbf{P}\mathbf{Q}_1^*$ components might be inconsistent. Furthermore we have shown that the $\mathbf{P}\mathbf{P}_1^*$
	 components of the initial condition are ``remembered'' by the time integration scheme, which is a key property for 
	 the Parareal algorithm considered next in this paper.

	\section{Parareal}\label{sec:parareal} 
	In the following section we will introduce the parallel-in-time method Parareal and study its application to differential
	algebraic equations.
	
	Let us consider the initial value problem of the quasilinear DAE \eqref{eq:quaslinDAE}. To apply the Parareal algorithm, first
	the time interval $\mathcal{I}$ is partitioned into $N$ smaller time windows 
	$\mathcal{I}_n = (T_{n-1},\, T_n]$ of size $\Delta T = (t_{\mathrm{end}}- t_0) / N$,
	with $T_0=t_0$ and $T_N = t_{\mathrm{end}}$. 
	In each iteration $k$ Parareal solves
	in parallel the $N$ initial value problems
	\begin{align}\label{eq:IVPPR}
		\mathbf{A}(\mathbf{x}_n,t)\mathbf{x}_n' + \mathbf{b}(\mathbf{x}_n,t) = 0, && \mathbf{x}_n(T_{n-1}) = \mathbf{X}^k_{n-1}
		&& \text{for }t\in\mathcal{I}_n\;,
	\end{align}
	with $\mathbf{X}^k_0 = \mathbf{x}_0$ and $n=1,\ldots,N$. This, however requires initial conditions for each subwindow $\mathcal{I}_n$,
	which  are a priori unknown. Therefore, the algorithm has to start with incorrect initial conditions at the interface points $T_n$.
	This yields jumps in the solution across subwindows, which Parareal
	tries to iteratively eliminate by updating the initial conditions. Thus, in addition to the parallel computation of the $N$
	initial value problems described previously, Parareal performs an update formula for the initial conditions which for the $(k+1)$th
	iteration reads \cite{Lions_2001aa,Gander_2007ac}
	\begin{align}\label{eq:PRupdate}
		\mathbf{X}_n^{k+1} = \mathcal{F}(T_n,T_{n-1},\mathbf{X}_{n-1}^k) + 
							\mathcal{G}(T_n,T_{n-1},\mathbf{X}_{n-1}^{k+1}) -
							\mathcal{G}(T_n,T_{n-1},\mathbf{X}_{n-1}^k)\;,
	\end{align}
	for $n=1,\ldots,N$.
	Here, $\mathcal{F}$ and $\mathcal{G}$  are solution operators of the initial value problem called 
	the fine and coarse propagator, respectively. 
	
	The fine propagator
	$\mathcal{F}(T_n,T_{n-1},\mathbf{X}_{n-1}^k)$ returns the solution of \eqref{eq:IVPPR} at time step $T_n$. It solves the problem 
	in a very accurate way and is thus
	computationally expensive to apply. However, as for the $(k+1)$th iteration  this solution only requires the initial condition computed
	at the previous Parareal iteration $\mathbf{X}_{n-1}^k$, it is performed in parallel. Therefore, simulation time
	is still reduced with respect to a sequential computation. 
	
	The second operator,
	$\mathcal{G}(T_n,T_{n-1},\star)$ gives the solution of the initial value problem at time point $T_n$ with initial condition $\star$
	at $T_{n-1}$. However, $\mathbf{X}_{n-1}^{k+1}$ is a solution of the current Parareal iteration and thus can not be computed in parallel.
	Therefore, the coarse solver has to be applied sequentially. To ensure computation time is reduced as much as possible,
	this operator has to be cheap to compute and as a consequence less accurate. Here, for example, a larger time step size can be employed
	or a reduced system can be solved.
	
	Combining the parallel computation of the expensive and accurate fine propagator with the sequential computation of the cheap 
	coarse propagator yields each Parareal iteration to require less computation time than the sequential simulation approach.
	This establishes a link to other parallelisation methods such as multiple shooting methods (see \cite{Gander_2015aa} for an historical overiew) 
	or multigrid approaches \cite{Falgout_2014aa}.
	
	\subsection{Parareal for DAEs}\label{sec:PRDAE}
	
	The application of the Parareal algorithm for nonlinear ordinary differential equations and its convergence is already studied \cite{Gander_2008aa}.
	However, for the case of higher index DAEs, its convergence has not been studied yet and its applicability is not ensured. 
	Here, especially for nonlinear DAEs, it can happen that the update formula \eqref{eq:PRupdate} yields an inconsistent
	initial condition for the fine solver, which may lead to divergence of the algorithm, slower
	convergence or an incorrect solution.
	Nevertheless, the algorithm has been applied to DAEs previously.  For example in \cite{Cadeau_2011aa} the algorithm is applied to a system of
	DAEs without a special handling.  In \cite{Schops_2018aa} it is 
	applied to index 1 DAEs with a special structure and in \cite{Falgout_2019aa} with a modified Parareal algorithm. Both approaches
	are special cases of the theory that is given in this paper.
	
	A modification of the Parareal algorithm  to be applied to quasilinear index 2 DAEs is presented here. The method
	is similar to the one given in \cite{Lamour_1994aa,Lamour_1997aa} for multiple shooting methods. There it is ensured that the 
	algorithm works by means of extending the Jacobian that is computed to update the initial conditions to avoid it being singular.  
	This is achieved by including an equation for the calculation of consistent initial conditions. 
	The idea behind the method here is to only apply the Parareal algorithm on the purely differential components of the DAE 
	and then compute the rest of the degrees of freedom accordingly. For that,  the update formula \eqref{eq:PRupdate} is
	restricted to
	\begin{align}\label{eq:PRupdateDAE}
		\mathbf{\hat{X}}_n^{k+1} = \mathbf{P}\mathbf{P}_1(\mathbf{\tilde{X}}_n^k, T_n)\mathbf{\tilde{X}}_n^k +
								\mathbf{P}\mathbf{P}_1(\mathbf{\bar{X}}_n^{k+1}, T_n)\mathbf{\bar{X}}_n^{k+1} -
								\mathbf{P}\mathbf{P}_1(\mathbf{\bar{X}}_n^{k}, T_n)\mathbf{\bar{X}}_n^{k}\;,
	\end{align}
	where $\mathbf{\tilde{X}}_n^k \coloneqq \mathcal{F}(T_n,T_{n-1},\mathbf{X}_{n-1}^k)$ and
	 $\mathbf{\bar{X}}_n^{k}\coloneqq\mathcal{G}(T_n,T_{n-1},\mathbf{X}_{n-1}^k)$. The resulting solution
	$\mathbf{\hat{X}}_n^{k+1}$ is then used to compute the consistent initial condition $\mathbf{X}_n^{k+1}$ such that
	\begin{equation}\label{eq:PP1=0}
		\mathbf{P}\mathbf{P}_1(\mathbf{X}_n^{k+1}, T_n)\left(\mathbf{X}_n^{k+1} - \mathbf{\hat{X}}_n^{k+1}\right)= 0\;.
	\end{equation}
	This can either be done analytically for simple DAEs  or with numerical techniques
 	as proposed e.g. in \cite{Estevez-Schwarz_2000ab,Estevez-Schwarz_2018ab}. The algorithm of \cite{Estevez-Schwarz_2018ab} is implemented in 
 	a Python package
 	called InitDAE\footnote{\url{https://www2.mathematik.hu-berlin.de/~lamour/software/python/InitDAE/html/InitDAE_Integration2020_3_7/}}, 
 	which is able to numerically compute consistent initial conditions for DAEs.
 	
 	Note that is approach can also be performed on DAEs with index higher than two. In that case, the projectors to extract the purely differential
 	components in \eqref{eq:PRupdateDAE} and \eqref{eq:PRupdateDAE} have to be adapted accordingly (see \cite{Lamour_2013aa}). 
 	
 	\subsubsection{Implicit Euler as propagator}\label{sec:eulpr}
	Note that, as the update \eqref{eq:PRupdateDAE} in the Parareal algorithm is performed sequentially, the computation of the 
	initial conditions $\mathbf{X}_n^{k+1}$ out of the obtained solution after the update $\mathbf{\hat{X}}_n^{k+1}$ is also 
	done in a sequential manner. If this operation is computationally expensive, then it can considerably increase the 
	simulation time of the Parareal algorithm. Furthermore, one of the advantages of the Parareal method is that, as it is not intrusive, 
	i.e., it can
	even be applied to black box simulators (as long as you can prescribe initial values). In such cases,  
	obtaining the explicit matrices
	of the DAE system that is solved might not be possible. These two inconveniences can be overcome for DAEs with a specific structure
	by means of using the implicit Euler method as a time integrator.
	
	\begin{proposition}[Parareal with implicit Euler]\label{prop:preul}
		We consider a DAE with structure as in \eqref{eq:DAElinindex2} and the corresponding projectors fulfilling Assumptions \ref{ass:canonical}, 
		\ref{ass:constT} and \ref{ass:Q1ct} and apply Parareal with the implicit Euler scheme on the coarse and the fine level.
		If
		\begin{itemize}
			\item for index~1 at least one time step is used on the fine level,
			\item for index~2 at least two time steps are used on the fine level,
		\end{itemize}
	 Parareal converges without the requirement of explicitly making the initial conditions consistent.
	\end{proposition}
	Let us consider the DAE  \eqref{eq:DAElinindex2} fulfilling
	Assumptions \ref{ass:constT} and \ref{ass:Q1ct}, as required in Proposition~\ref{pro:impleul}. Here, a 
	constant projector $\mathbf{P}\mathbf{P}_1^*$ onto the differential components exists, and
	thus the classic Parareal update
	\begin{align*}
		\mathbf{\hat{X}}_n^{k+1} = \mathbf{\tilde{X}}_n^k +\mathbf{\bar{X}}_n^{k+1} - \mathbf{\bar{X}}_n^{k}
	\end{align*}
	implies
	\begin{align*}
	\mathbf{P}\mathbf{P}_1^*\mathbf{\hat{X}}_n^{k+1} = \mathbf{P}\mathbf{P}_1^*\mathbf{\tilde{X}}_n^k +
	\mathbf{P}\mathbf{P}_1^*\mathbf{\bar{X}}_n^{k+1} - \mathbf{P}\mathbf{P}_1^*\mathbf{\bar{X}}_n^{k}\;.
	\end{align*}
	Implicit Euler yields a consistent solution after at most two
	 time steps (see Proposition~\ref{pro:impleul}). Furthermore, its solution corresponds to the one obtained with a consistent initial condition where
	 the $\mathbf{P}\mathbf{P}_1^*$ coincide, that is,
	 \begin{align*}
	 	\mathbf{P}\mathbf{P}_1^*\left(\mathbf{X}_n^{k+1} - \mathbf{\hat{X}}_n^{k+1}\right) = 0\;.
	 \end{align*}
 	Thus, if on the fine level at least two Euler steps are performed, the solution at the end of the interval is 
 	equivalent to first computing the consistent initial condition $\mathbf{X}_n^{k+1}$ and then 
 	 starting the simulation with it. On the coarse level, already one implicit Euler step suffices, as only the
 	 $\mathbf{P}\mathbf{P}_1^*$ components are relevant for the update and those are remembered by the time integration
 	 method also in the first time step, as shown in the proof of Proposition~\ref{pro:impleul}.
	
	Therefore, for DAEs fulfilling the requirements of Proposition~\ref{pro:impleul} it is not necessary to 
	use the modified Parareal update \eqref{eq:PRupdateDAE} and the subsequent computation of a consistent 
	initial condition. Here, the usage of the implicit Euler as coarse propagator and performing at least
	two implicit Euler steps on the fine level are sufficient for the Parareal algorithm to converge.
	
	In \cite{Schops_2018aa} this property is exploited for an index~1 DAE solved with Parareal and the implicit Euler scheme.
	
	The approach taken in \cite{Falgout_2019aa}, however, is a special case of the algorithm proposed in Section~\ref{sec:PRDAE}
	for index~1 DAEs. There, all the components are added in the Parareal update formula \eqref{eq:PRupdate}, which is equivalent 
	to \eqref{eq:PRupdateDAE} if the projector $\mathbf{P}$ is constant and the algebraic variables are made consistent afterwards.
	Both cases are covered within the formalised results of this paper, and the theory is expanded to index~2 DAEs.
	
	\section{Numerical Examples}\label{sec:numerics}
	In the following, two nonlinear index 2 differential algebraic equations are solved with two versions of
	the Parareal algorithm, the classic one and the modified Parareal algorithm for DAEs from Section~\ref{sec:PRDAE}.
	
	\subsection{Nonlinear Index~2 DAE}
	 To test the proposed algorithm, we first consider the following index 2 DAE as a toy example
	 \begin{subequations}\label{eq:daeex1}
	 	\begin{align}
	 	x_0' + g(x_2) 	&= 0\\
	 	x_1' - x_2		&=0 \\
	 	x_1-0.015\sin(2\pi 10 t) &=0\;,
	 	\end{align}
	 \end{subequations}
	 with degrees of freedom $\mathbf{x}^{\top} = (x_0,\, x_1,\, x_2)$ and nonlinear function
	 \begin{equation}
		 g(x) = 
		 \begin{cases} 0 &\mbox{if } x \leq 1\;, \\
		 e^{-(x-1)^{-2}} & \mbox{if } 1 < x \leq 2\;, \\
		 e^{-(x-1)^{-2}}-\frac{1}{8}e^{\frac{3}{4}}e^{-(x-2)^{-2}} & \mbox{otherwise.}
		 \end{cases}
	 \end{equation}
	 Here, the projector matrices $\mathbf{P}$, $\mathbf{P}_1(\mathbf{x},t)$ are 
	 \begin{align} \label{eq:PP1}
	 	\mathbf{P}  = 
	 	\begin{pmatrix}
	 	1 & 0 & 0\\
	 	0 & 1 & 0\\
	 	0 & 0 & 0
	 	\end{pmatrix}
	 	\quad \text{and} \quad
	 	\mathbf{P}_1(\mathbf{x},t)  = 
	 	\begin{pmatrix}
	 	1 & \frac{\partial}{\partial x_2}g(x_2) & 0\\
	 	0 & 0 & 0\\
	 	0 & -1 & 1
	 	\end{pmatrix}\;.
	 \end{align}
	 Note that the index 2 component $x_2$ appears nonlinearly in the DAE \eqref{eq:daeex1}. Therefore, this example does not fulfil 
	 the requirements for the consistentialisation of the implicit Euler scheme and
	 Proposition~\ref{prop:preul} does not necessarily hold. This implies that starting the Euler scheme with an inconsistent initial 
	 value $\mathbf{x}^0$ does not
	 necessarily yield the same solution after two time steps than starting with the consistent initial condition $\mathbf{x}_0$ with the 
	 same $\mathbf{P}\mathbf{P}_1(\mathbf{x}_0,t)$ components. 
	 
	 To exemplify this behaviour a counterexample is presented. Let us consider the inconsistent initial condition 
	 $\mathbf{x}^0 = (0,\, -1,\; 0)^{\top}$ at $t_0=0$. Its corresponding consistent initial condition $\mathbf{x}_0$  with
	 \begin{equation}\label{eq:prconsist}
	 	\mathbf{P}\mathbf{P}_1(\mathbf{x}_0,t)(\mathbf{x}_0-\mathbf{x}^0)=0
	 \end{equation}
	 is $\mathbf{x}_0 = (0,\; 0,\; 0.3\pi)^{\top}$. After two implicit Euler steps with e.g. time step $h=1/3$ and starting with the inconsistent
	 initial condition $\mathbf{x}^0$, $$x_0^2 = -g(0.045\sin(20\pi/3)+3)/3\neq0$$ is obtained. The same scheme
	 for initial value $\mathbf{x}_0$, however, yields the correct solution
	 $$x_{0,2} = 0$$ if initialised with the corresponding consistent value $\mathbf{x}_0$ and thus $x_0^2\neq x_{0,2}$.
	 
	 Note that this example
	 is an artificially created DAE without dynamics, due to the definition of the nonlinear function
	 $g(x)$ and the fact that $x_2$ is always $\leq 1$.
	 
	 To study the proposed modification of the Parareal algorithm, we apply the variants:
	 
	 	\begin{itemize}
	 	\item\textbf{PR Euler} \\
	 	In the first algorithm no special handling is implemented, that is, classic Parareal with implicit Euler
	 	as time integrator is applied with a small time step size $\delta t$ for the fine solver and a larger one $\Delta T$
	 	for the coarse propagator. 
	 	
	 	\item\textbf{PR Init}\\
	 	For the second simulation, in the update formula we only consider the $\mathbf{P}\mathbf{P}_1(\mathbf{x},t)$ components as in 
	 	\eqref{eq:PRupdateDAE}. Afterwards,
	 	the corresponding consistent initial condition with the same $\mathbf{P}\mathbf{P}_1(\mathbf{x},t)$ 
	 	components as in \eqref{eq:PP1=0} is computed. 
	 \end{itemize}
	 	For this particular example, the Parareal update \eqref{eq:PRupdateDAE}, with $\mathbf{P}\mathbf{P}_1$ as shown in
	 	\eqref{eq:PP1}, is
	 	\begin{align*}
	 	(\hat{x}_0)_n^{k+1} ={}& (\tilde{x}_0)_n^k + \frac{\partial }{\partial x}g(x)\Bigr\rvert_{(\tilde{x}_2)_n^k}(\tilde{x}_1)_n^k
	 	+ (\bar{x}_0)_n^{k+1} + \frac{\partial }{\partial x}g(x)\Bigr\rvert_{(\bar{x}_2)_n^{k+1}}(\bar{x}_1)_n^{k+1}\\
	 	&-(\bar{x}_0)_n^k - \frac{\partial }{\partial x}g(x)\Bigr\rvert_{(\bar{x}_2)_n^k}(\bar{x}_1)_n^k\;.
	 	\end{align*}
	 	Once the updated value $(\hat{x}_0)_n^{k+1}$ is obtained, the consistent initial condition is computed analytically
	 	with
	 	\begin{align*}
	 	(x_2)_n^{k+1} &= 0.3\pi\cos(20\pi T_n)\\
	 	(x_1)_n^{k+1} &= 0.015\sin(20\pi T_n)\,,
	 	\end{align*}
	 	and using \eqref{eq:prconsist} with projector matrices \eqref{eq:PP1} the differential variable is obtained
	 	\begin{align*}
	 	(x_0)_n^{k+1} &= (\hat{x}_0)_n^{k+1} +  \frac{\partial }{\partial x}g(x)\Bigr\rvert_{(x_2)_n^{k+1}}( 0 - (x_1)_n^{k+1} )\;.
	 	\end{align*}
 
 \begin{figure}
 	\centering
 	\begin{subfigure}{0.49\textwidth}
 		\centering
 		\includegraphics{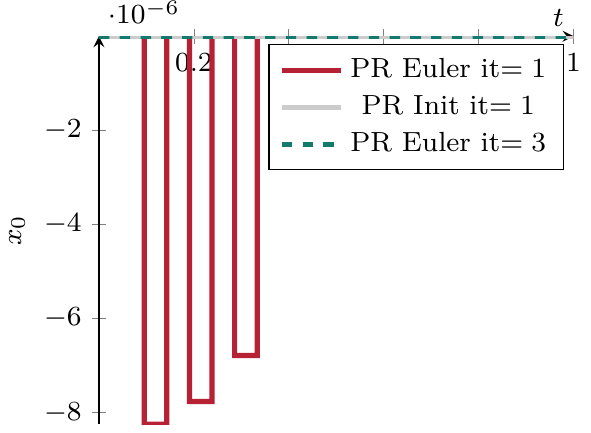}
 		\caption{Index 0 component $x_0$.}
 	\end{subfigure}%
 \begin{subfigure}{0.49\textwidth}
 	\centering
 	\includegraphics{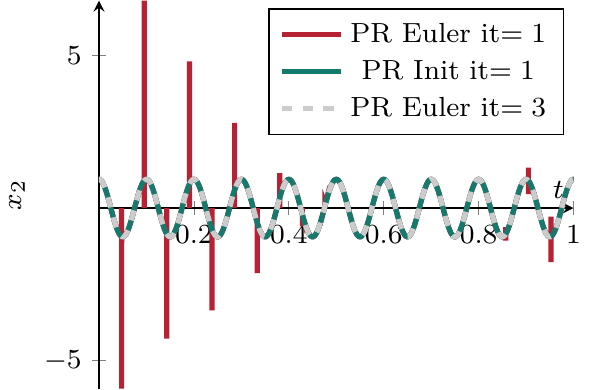}
 	\caption{Index 2 component $x_2$.}
 	\end{subfigure}
 \caption{Solution of the Parareal algorithm after the 1st and 3rd iterations for 
 	the classic algorithm `PR Euler' and the modified algorithm for DAEs `PR Init'.}\label{fig:daetoy}
 \end{figure}
	
	For both simulations $N= 21$ processors are chosen and the simulation time window $\mathcal{I} = [0\; 1)$. The time step
	size of the fine implicit Euler propagator is set to 
	$\delta t = 10^{-5}$ and the coarse solver is chosen to perform
	one time step per window and thus has time step size $\Delta T = 1/N$.
	The Parareal algorithm is iterated until the $l_2$ norm of the difference between 
	the $\mathbf{P}\mathbf{P}_1$ components of the solution at the end of interval $\mathcal{I}_{n-1}$ and
	the initial condition of $\mathcal{I}_{n}$ for all $T_n$  is below a relative tolerance of $5\cdot 10^{-4}$ 
	and an absolute tolerance of $10^{-10}$ (see error norm of \cite{Hairer_2000aa}). 
	
	Whereas the first algorithm (PR Euler) requires 3 iterations to reach the required tolerance, the second algorithm (PR Init)
	converges immediately after the 1st Parareal iteration. The obtained solutions for the purely differential
	component $x_0$ and the index two variable $x_2$ are depicted in Figure~\ref{fig:daetoy}. Here it can be seen that
	the proposed algorithm for DAEs `PR Init' obtains the correct solution for the index two variable $x_2$ already at the first
	iteration. Therefore, the algorithm converges immediately, whereas the classic Parareal algorithm has jumps on $x_2$ that yield
	an incorrect solution also for the differential component $x_0$. After 3 Parareal iterations the classic Parareal algorithm
	manages to reduce the jumps on both $x_2$ and $x_0$ but this is not covered by theory.

	 Even in this simple case without dynamics, where the 
	Parareal algorithm should converge immediately, the classic algorithm requires 3 iterations due to the inconsistency of the
	index~2 variable, which introduces large errors in the nonlinearity term affecting thus the solution of the entire system.

	\subsection{Circuit with Modified Nodal Analysis}
	To exemplify the theoretical results presented in Sections~\ref{sec:dae} and~\ref{sec:eulpr} as well as
	demonstrate that the assumptions that are taken are realistic for real-life applications, we further
	present an example arising from the description of an electric network.
	
	\subsubsection{Modified Nodal Analysis}\label{sec:mna}
	We consider electric networks containing capacitors (C), inductors (L), resistors (R) and voltage (V)
	and current (I) sources.
	In modified nodal analysis, networks are described by means of incidence matrices 
	$\mathbf{A}_{\star}$, $\star\in\{\mathrm{C,L,R,V,I}\}$ that characterise the branch-to-node relation of the underlying graph
	for the corresponding elements.
	Applying Kirchhoff's current law and the lumped parameter models of the different components, we obtain
	the following system of DAEs \cite{Estevez-Schwarz_2000aa,Gunther_2005aa}
	\begin{subequations}\label{eq:fluxchargemna}
		\begin{align}
		\mathbf{A}_{\mathrm{C}}\mathbf{q}' + \mathbf{A}_{\mathrm{R}}\mathbf{g}_{\mathrm{R}}(\mathbf{A}_{\mathrm{R}}^{\top}\mathbf{e},t) +
		\mathbf{A}_{\mathrm{L}}\mathbf{i}_{\mathrm{L}} + \mathbf{A}_{\mathrm{V}}\mathbf{i}_{\mathrm{V}} +
		\mathbf{A}_{\mathrm{I}}\mathbf{i}_{\mathrm{s}}(t)&=0\\
		\mathbf{q}-\mathbf{q}_{\mathrm{C}}(\mathbf{A}_{\mathrm{C}}^{\top}\mathbf{e},t) &= 0\\
		\boldsymbol{\phi}'-\mathbf{A}_{\mathrm{L}}^{\top}\mathbf{e}&=0\\
		\boldsymbol{\phi} - \boldsymbol{\phi}_{\mathrm{L}}(\mathbf{i}_{\mathrm{L}},t) &=0\\
		\mathbf{A}_{\mathrm{V}}^{\top}\mathbf{e}-\mathbf{v}_{\mathrm{s}}(t)&=0\;.
		\end{align}
	\end{subequations}
	Here the system of DAEs is given in the flux-charge formalism. 
	In this formulation, the additional degrees of freedom are $\mathbf{e}:\mathcal{I}\rightarrow\mathbb{R}^{n_{\mathrm{e}}}$, 
	the vector of node potentials, 
	$\mathbf{i}_{\star}:\mathcal{I}\rightarrow\mathbb{R}^{n_{\star}}$, the vector of currents through branches
	containing the element $\star$, $\mathbf{q}:\mathcal{I}\rightarrow\mathbb{R}^{n_{\mathrm{C}}}$, the vector
	of charges in capacitances and $\boldsymbol{\phi}:\mathcal{I}\rightarrow\mathbb{R}^{n_{\mathrm{L}}}$, the vector
	of fluxes in inductances. Finally, $\mathbf{q}_{\mathrm{C}}(\cdot)$, $\boldsymbol{\phi}_{\mathrm{L}}(\cdot)$,
	$\mathbf{g}_{\mathrm{R}}(\cdot)$, $\mathbf{i}_{\mathrm{s}}(\cdot)$ and $\mathbf{v}_{\mathrm{s}}(\cdot)$ are
	(nonlinear) functions describing the lumped parameter relations for the different elements. The vector of node potentials
	and the incidence matrices allow extracting the voltages across the branches containing a given element $\star$ with the relation
	$\mathbf{v}_{\star} = \mathbf{A}_{\star}^{\top}\mathbf{e}$.
	
	The tractability index of this system has already been analysed (see e.g. \cite{Estevez-Schwarz_2000aa}) and is in the worst case 2.
	This result is given by only topological properties of the underlying graph.

	Note that, unlike in the classic MNA formulation, the flux-charge system \eqref{eq:fluxchargemna} yields a system of
	DAEs with a constant mass matrix. This is achieved thanks to the introduction of the degrees of freedom 
	$\mathbf{q}$ and $\boldsymbol{\phi}$.
	
	The index analysis in \cite{Estevez-Schwarz_2000aa} shows that the possible index 2 components of the system are
	currents through voltage sources $\mathbf{i}_{\mathrm{V}}$ and voltages across inductances 
	$\mathbf{A}_{\mathrm{L}}^{\top}\mathbf{e}$. These two degrees of freedom appear linearly in the original 
	system \eqref{eq:fluxchargemna} and thus flux-charge MNA has linear index 2 components \cite{Baumanns_2010aa}. Therefore, 
	system \eqref{eq:fluxchargemna} is, in the worst case, an index 2 tractable DAE with linear index 2 components
	and constant mass matrix as described in \eqref{eq:DAElinindex2}. Finally, in \cite{Estevez-Schwarz_2000aa} it is also shown that
	Assumption~\ref{ass:Q1ct} is fulfilled, as $\ima \mathbf{Q}_1(\mathbf{U}\mathbf{x},t)$ is constant. This allows
	the application of Proposition~\ref{pro:impleul} to the system of DAEs obtained from flux-charge MNA and thus
	the implicit Euler scheme returns a consistent solution after at most two time steps even if an inconsistent
	initial condition is given.

	\subsubsection{Example}
	The second nonlinear index 2 DAE arises from a circuit described
	with flux-charge modified nodal analysis (see Section~\ref{sec:mna}). 
	We consider the nonlinear index 2 circuit of Figure~\ref{fig:circuit} 
	with linear parameters $R_{1,1} = 10^{-2}\,\Omega$, $R_{1,2} = 10^{-2}\,\Omega$, $L_1 = 10^{-4}\,$H and
	current source $$i_1(t) = (100\sin(100\pi t) +  50\sin(400\pi t))\,\text{A}\;.$$ For the nonlinear inductance the model of
	\cite{Capua_2016aa} is used with nominal inductance $L_{\mathrm{nom}} = 10^{-3}\,$H,
	deep saturation inductance $L_{\mathrm{deepsat}} = 8\cdot10^{-4}$, smoothness factor $\sigma = 5\cdot 10^{-2}$
	and current $I_{\mathrm{L}}^* = 90\,$A.
	\begin{figure}
		\centering
		\includegraphics{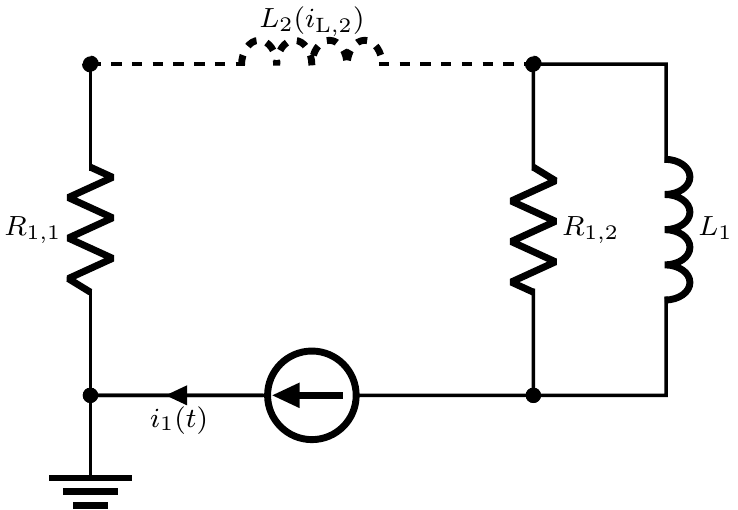}
		\caption{Index 2 circuit with nonlinear inductance $L_2(i_{\mathrm{L},2})$ as described in \cite{Capua_2016aa}.
				Figure based on \cite{Cortes-Garcia_2020ae}.}\label{fig:circuit}
	\end{figure}

	As it has been explained in Section~\ref{sec:mna},
	the system of DAEs arising from flux-charge MNA fulfills the requirements to apply Proposition~\ref{pro:impleul}. Thus,
	following the idea of Section~\ref{sec:eulpr}, we do not require any special handling and can apply classic Parareal to that system
	of equations without any drawback. 	
	
	To exemplify this, we apply again Parareal twice. Once the classic version `PR Euler' and the second version designed for DAEs
	`PR Init'. Unlike in the first example, the consistent initial conditions are not
	computed manually, but numerically with the Python package InitDAE. For the `PR Init'
	algorithm, the projectors $\mathbf{P}\mathbf{P}_1*$ are used for the update formula.
	In both simulations $N=15$ processors are chosen and the simulation time window is set to 
	$\mathcal{I}= [0,\;0.2)$. 
	For the fine solution implicit Euler with a time step size of $\delta t = 10^{-5}$ is used and the coarse
	implicit Euler solver performs one time step per window. 
	The initial condition $\mathbf{x}_0$ is computed by starting an implicit Euler scheme with an inconsistent value $\mathbf{x}^{-2}=0$ at time step
	$t_0-2\delta t$ and performing two steps until $t_0$. Exploiting Proposition~\ref{pro:impleul}, the obtained solution $\mathbf{x}_0$ for the
	rest of the simulation is 
	a consistent initial condition.
	The error is computed as in the previous example and
	the relative tolerance is set to $10^{-4}$, whereas the absolute tolerance is chosen to be $10^{-8}$. 
	
	This time
	both algorithms require 4 Parareal iterations to reach the required tolerance. The result is not surprising, 
	as both algorithms supposed to be  doing the same on the differential, and thus transient part, of the system and the convergence depends only 
	on the problem.

	\begin{figure}
		\centering
		\begin{subfigure}{0.48\textwidth}
			\centering
			\includegraphics{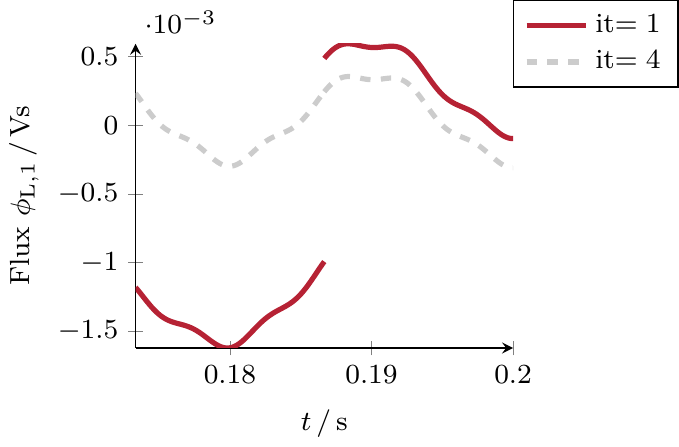}
			\caption{Solution with the classic Parareal algorithm and implicit Euler as time integrator.}
		\end{subfigure}\hfill
	\begin{subfigure}{0.48\textwidth}
		\centering
		\includegraphics{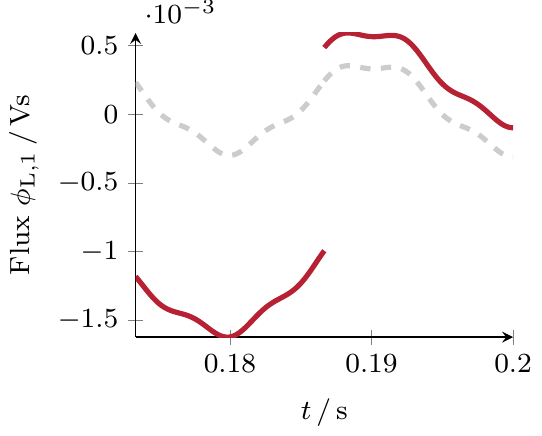}
		\caption{Solution with the modified Parareal algorithm for DAEs.}
	\end{subfigure}
	\begin{subfigure}{0.48\textwidth}
	\centering
	\includegraphics{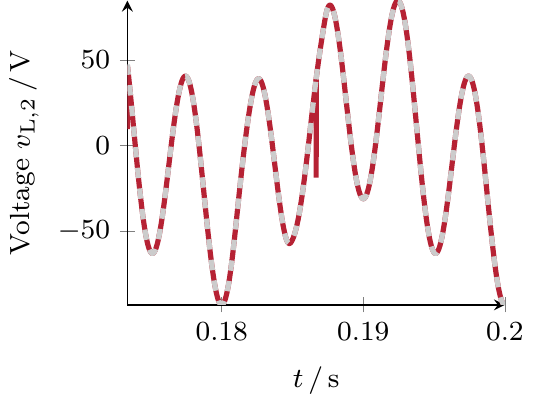}
	\caption{Solution with the classic Parareal algorithm and implicit Euler as time integrator.}
	\end{subfigure}\hfill
	\begin{subfigure}{0.48\textwidth}
	\centering
	\includegraphics{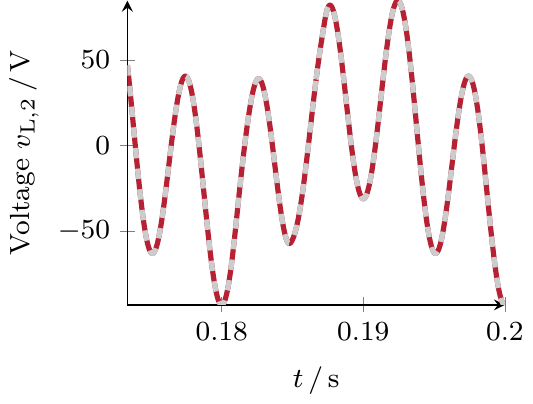}
	\caption{Solution with the modified Parareal algorithm for DAEs.}
	\end{subfigure}
		\caption{Solution of the Parareal algorithm for the last two windows $\mathcal{I}_{14}$ and $\mathcal{I}_{15}$.}\label{fig:circsol}
	\end{figure}
		
	In Figure~\ref{fig:circsol}, the purely differential component $\phi_{\mathrm{L},1}$ and the index 2 component
	$v_{\mathrm{L,2}}$
	of the fine solutions on the last two windows $\mathcal{I}_n$, $n=14,15$ are depicted
	for  the first and the last Parareal iterations. It can be seen that in both cases the purely
	differential component has a jump at the first iteration and becomes continuous at the last one, which is a typical
	behaviour of the Parareal algorithm. As in the previous example, the index 2 component is immediately smooth and converged at the first iteration
	for the `PR Init' algorithm, whereas `PR Euler' starts with an inconsistent solution at the first step in the first iteration. Here
	the behaviour of the implicit Euler scheme when starting with inconsistent initial conditions
	 can be observed: the solution jumps to the correct value due to 
	Property~\ref{pro:impleul}. This, however, does not negatively affect  the convergence of `PR Euler', as the purely differential components
	are handled equally in both algorithms.
	\begin{remark}
		Note that, even for a DAE fulfilling the requirements of Proposition~\ref{pro:impleul}, it can happen that `PR Init'  
		may reach the required tolerance for the differential components, while the algebraic ones have not reached the required accuracy yet.
		Let us consider the solution at $\mathcal{I}_n$.
		For example, for index~1 components it suffices to take the result of the 
		last time step of the previous window $\mathcal{I}_{n-1}$ instead of the initial value at the beginning of
		$\mathcal{I}_n$ to ensure a consistent value at time $T_{n-1}$. 
		The index~2 components would in addition require to either ignore the first (possibly)
		inconsistent time step of $\mathcal{I}_n$ at $T_{n-1}+\delta t$ (with $\delta t$ being the time step size of the fine propagator), or perform
		one extra time step of the solution of $\mathcal{I}_{n-1}$ to arrive at the time $T_{n-1}+\delta t$.
	\end{remark}

	\section{Conclusions}\label{sec:conclusions}
	This article has presented a modification of the Parareal algorithm for its application to quasilinear index 2 tractable DAEs.
	Its extension to higher index systems requires extra care, however it follows analogously from the projector-based decoupling
	of differential algebraic equations. For a large class of DAEs i.e. linear index~2 components and constant mass matrix
	as given in flux-charge formulated modified nodal analysis,
	a new property of the implicit Euler scheme is proven. This property allows the usage of the classic Parareal algorithm, as long
	as the implicit Euler scheme is used as the time integrator of both the first two time steps of the fine as well as for the coarse propagator.
	
	The theoretical results are backed up by numerical simulations of two DAEs, one toy example with nonlinear index 2 components and the other one arising 
	from a flux-charge modified nodal analysis formulated circuit. As theoretically expected, the modified Parareal algorithm speeds up the convergence when
	applied to a DAE with nonlinear index~2 components.
	
	\subsection*{Acknowledgement}
	We would like to thank Diana Est\'evez Schwarz and Lennart Jansen for the fruitful discussions as well as Pia Callmer 
	for her assistance in implementing the algorithm. This work is based on Chapter~5 of the PhD thesis of 
	Idoia Cortes Garcia \cite{Cortes-Garcia_2020ae}.

	This work is supported by the Graduate School CE within the Centre for Computational Engineering at Technische Universität Darmstadt and DFG Grant SCHO1562/1-2 and BMBF Grant 05M2018RDA (PASIROM).

\end{document}